\documentclass[10pt]{article}

\usepackage{amsmath,amssymb,amsfonts,amsthm}
\usepackage{palatino}
\usepackage[euler-digits]{eulervm}
\usepackage{microtype}
\usepackage{tikz,tikz-cd}
\usepackage{graphicx}
\usepackage{ytableau}
\usepackage{fullpage}

\usepackage{authblk}

\usetikzlibrary{positioning, fit, shapes.geometric} % USE FIT TO CIRCLE MULTIPLE NODES

\usepackage{amsthm}
\theoremstyle{plain}
\newtheorem{thm}{Theorem}[subsection]
\newtheorem{prop}{Proposition}[subsection]
\newtheorem{lemma}{Lemma}[subsection]
\newtheorem{cor}{Corollary}[subsection]
\newtheorem*{conj}{Conjecture}

\theoremstyle{definition}
\newtheorem{defn}{Definition}[subsection]
\newtheorem{example}{Example}[subsection]

\definecolor{darkblue}{rgb}{0,0,0.7} % darkblue color
\definecolor{darkred}{rgb}{0.7,0,0} % darkred color
\definecolor{darkgreen}{rgb}{0,0.7,0} % darkgreen color
\definecolor{darkpurple}{rgb}{0.5,0,0.7}

\newcommand\ZZ{\mathbb{Z}}

\newcommand\Acal{\widehat{\mathcal{F}}}
\newcommand\Bcal{\mathcal{B}}
\newcommand\Ecal{\mathcal{E}}
\newcommand\Fcal{\mathcal{F}}
\newcommand\Ical{\mathcal{I}}

\newcommand{\Addresses}{{% additional braces for segregating \footnotesize
  \bigskip
  \footnotesize
  \textsc{Department of Mathematics, University of Southern California, 3620 S. Vermont Ave., Los Angeles, CA 90089}\par\nopagebreak
  \textit{E-mail address:} \texttt{armon@usc.edu}
}}

\title{Extremal subsets and atom-positivity}
\author{Sam Armon}
\date{}

\begin{document}

\maketitle

\begin{abstract}
Demazure crystals give a combinatorial framework in which to study Demazure modules. They are \emph{extremal}, in that they satisfy Kashiwara's string property, and they are \emph{Demazure atom-positive}, in that they decompose naturally into subsets called crystal Demazure atoms. The properties of extremality and atom-positivity are further linked by a conjecture of Polo, which concerns the structure of the tensor product of two Demazure modules. We study these two properties in isolation --- first providing a recursive formula for generating crystal Demazure atoms which generalizes a result of Lascoux-Sch\"{u}tzenberger, and then examining the structure of extremal subsets and their characters --- before commenting on the connection (or lack thereof) between them.
\end{abstract}

\section{Introduction}
The finite-dimensional simple modules $V(\lambda)$ of a (connected, simply connected) complex semisimple Lie group $G$ are parameterized by dominant weights $\lambda \in P^+$. The representation theory of $G$ is fairly well-understood; in particular, for $\lambda, \mu \in P^+$, the tensor product $V(\lambda) \otimes V(\mu)$ admits a \emph{good filtration}; that is, a filtration in which each successive quotient is isomorphic to $V(\nu)$ for some $\nu \in P^+$. The theory of crystal bases, developed independently by Kashiwara \cite{Kas90} and Lusztig \cite{Lus}, provides a combinatorial model for the representation theory of $G$. In particular, there exists for each $\lambda \in P^+$ a \emph{highest weight crystal} $\Bcal(\lambda)$ --- which one may regard as a directed, weighted graph --- whose character agrees with $V(\lambda)$, and the tensor product of highest weight crystals $\Bcal(\lambda) \otimes \Bcal(\mu)$ decomposes as a direct sum of highest weight crystals, witnessing the good filtration of $V(\lambda) \otimes V(\mu)$.

For $B \subseteq G$ a Borel subgroup, the \emph{Demazure module} $V_w(\lambda) \subseteq V(\lambda)$ is a cyclic $B$-submodule indexed by an element $w \in W$ of the Weyl group, generated by the (unique, up to a scalar) extremal weight vector of weight $w \cdot \lambda$ in $V(\lambda)$ \cite{Dem74a}. Generalizing the Borel-Weil-Bott theorem, $V_w(\lambda)$ may be identified with $H^0(X_w,\mathcal{L}_{-\lambda})$, the space of sections of a line bundle $\mathcal{L}_{-\lambda}$ associated to the antidominant weight $-\lambda$ over a Schubert variety $X_w \subseteq G/B$. On the level of crystals, Kashiwara \cite{Kas93} and Littelmann \cite{Lit} construct a subset $\Bcal_w(\lambda) \subseteq \Bcal(\lambda)$ whose character agrees with the Demazure module $V_w(\lambda)$, and this character is computed via the refined Demazure character formula \cite{Jos}. 

A $B$-module is said to have an \emph{excellent filtration} if each successive quotient in the filtration is isomorphic to $V_w(\lambda)$ for some $\lambda \in P^+$ and $w \in W$. The tensor product $V_v(\lambda) \otimes V_w(\mu)$ does not, in general, admit an excellent filtration, and correspondingly, $\Bcal_v(\lambda) \otimes \Bcal_w(\mu)$ does not always decompose as a direct sum of Demazure crystals. A conjecture of Polo \cite{Polo} states, however, that this tensor product should always admit a more general $B$-module filtration called a \emph{Schubert filtration}, whose successive quotients are isomorphic to a $B$-module of the form $H^0(S,\mathcal{L}_\lambda)$, where $S$ is a \emph{union} of Schubert varieties. We study two crystal-theoretic properties --- atom-positivity, and extremality --- related to Polo's conjecture, which we will now briefly outline.

In type $A$, Demazure characters form a basis for the polynomial ring $\ZZ[t_1,\ldots,t_n]$ which lifts the basis of Schur polynomials for the ring of symmetric polynomials. Polynomial Demazure atoms --- originally studied by Lascoux-Sch\"{u}tzenberger \cite{LS} --- decompose the Demazure characters, in that Demazure characters are positive sums of polynomial Demazure atoms. These polynomials were shown by Mason \cite{Mason} to be certain specializations of nonsymmetric Macdonald polynomials in type $A$, and \cite{BBBG} construct a solvable lattice model whose partition function equals the Demazure atom. The product of Demazure characters cannot, in general, be written as a positive sum of Demazure characters, but Polo's conjecture suggests that this product should expand positively in the basis of Demazure atoms; this more general conjecture was also directly proposed by Pun \cite{Pun}, who solved it in the special case that the indexing compositions each have length $\le 3$.

For each $\lambda \in P^+$ and $w \in W$, Kashiwara \cite{Kas02} defines a subset $\mathcal{A}_w(\lambda) \subseteq \Bcal(\lambda)$ called a crystal Demazure atom which lifts the above decomposition to the level of crystals; that is, Demazure crystals naturally decompose into disjoint unions of crystal Demazure atoms, and (at least in type $A$) the character of a crystal Demazure atom is a polynomial Demazure atom. Particular crystal subsets which play a vital role in Polo's conjecture have the \emph{atom-positivity} property --- that is, they decompose into a disjoint union of crystal Demazure atoms. To this end, we define a new family of operators $\{ \Acal_i : i \in I \}$, and prove that they may be used to recursively generate entire crystal Demazure atoms in Theorem \ref{thm:atomic-op}. This theorem generalizes a result of Lascoux-Sch\"{u}tzenberger in type $A$ to arbitrary finite type, and provides a global alternative to the right key map of Lascoux-Sch\"{u}tzenberger, which was later generalized by Jacon-Lecouvey \cite{JL} and Santos \cite{Santos}.

The second property we consider is \emph{extremality}; this is one of the three defining properties of Demazure crystals found by Kashiwara, and in particular is the only property which does not directly follow from by necessity (see Theorem \ref{thm:dem-def}(3)). While $\Bcal_v(\lambda) \otimes \Bcal_w(\mu)$ does not, in general, decompose as a direct sum of Demazure crystals, Kouno \cite{Kouno} found a global condition which is necessary and sufficient for this decomposition to hold. A recent theorem of \cite{ADG} replaces Kouno's global condition with the local condition of extremality; that is, $\Bcal_v(\lambda) \otimes \Bcal_w(\mu)$ is a direct sum of Demazure crystals if and only if the tensor product is extremal. We prove that extremal subsets are generated by their lowest weight elements, and furthermore that an extremal subset generated by a single lowest weight element is uniquely determined by its character in Proposition \ref{prop:ext-char}. We also examine the relationship between extremality and atom-positivity due to their shared relationship with Polo's conjecture, but find that the relationship is not so direct; in particular, we construct an extremal subset whose character is not a positive sum of polynomial Demazure atoms, and an atom-positive subset which is not extremal.

The paper is structured as follows. In Section \ref{sec:background} we review the necessary facts about semisimple Lie groups, and then introduce highest weight crystals and Demazure crystals. In Section \ref{sec:atoms} we define crystal Demazure atoms and introduce a new family of operators which allows one to recursively construct these subsets in the crystal, before discussing the notion of (strong and weak) atom-positivity. Finally, in Section \ref{sec:extremal} we define extremal subsets and study their characters.

\section{Crystals}\label{sec:background}
Let $G$ be a connected, simply connected, complex semisimple Lie group, with $B \subseteq G$ a Borel subgroup and $T \subseteq B$ a maximal torus. The adjoint action of $T$ on the associated Lie algebra $\mathfrak{g}$ determines a root system $\Phi$; we let $I$ denote the index set of the corresponding Dynkin diagram, and $P \supseteq \Phi$ the weight lattice of $\mathfrak{g}$. Let $\{ \alpha_i \}_{i \in I} \subset \Phi$ denote the set of simple roots, $\{ \alpha_i^\vee \}_{i \in I} \subset \Phi^\vee$ the set of simple coroots, and $\{ s_i \}_{i \in I}$ the corresponding set of simple reflections.

The Weyl group $W$ is generated by $\{ s_i \}_{i \in I}$. If $w = s_{i_k} \cdots s_{i_1}$ for some $i_1, \ldots, i_k \in I$, where $k$ is minimal among all such expressions for $w \in W$, then we say that $s_{i_k} \cdots s_{i_1}$ is a \emph{reduced word} for $w$, and that the \emph{length} of $w$ is $\ell(w) = k$. The \emph{Bruhat order} on $W$ is the partial order $\le$ determined by the \emph{subword property}: for $v, w \in W$, we have $v \le w$ if and only if, for any reduced word for $w$, there exists a subword which is a reduced word for $v$.

The Weyl group $W$ acts on the weight lattice $P$. For any $\lambda \in P^+$, let $W_\lambda \le W$ denote the stabilizer subgroup of $\lambda$. Many of the constructions in this paper depending on $w \in W$ will depend only on the left coset of $w$ in $W / W_\lambda$ for a given $\lambda \in P^+$. To this end, we let $W^\lambda$ denote the set of minimal-length coset representatives in $W / W_\lambda$.

%There is a natural partial order $\preccurlyeq$ on $P$ given by $\lambda \preccurlyeq \mu$ if and only if $\mu - \lambda = \sum_{i \in I} c_i \alpha_i$ for $c_i \in \ZZ_{\ge 0}$. 

The cone $P^+\subset P$ of dominant weights is given by
\[
P^+ = \{ \lambda \in P : \langle \lambda, \alpha_i^\vee \rangle \ge 0 \text{ for all } i \in I \},
\] 
and the simple $G$-modules $V(\lambda)$ are parameterized by $\lambda \in P^+$. Let $\mathcal{O}$ denote the free ableian group generated by $\{ t^\beta : \beta \in P \}$. Each $G$-module $V$ decomposes into weight spaces $V \cong \bigoplus_{\beta \in P} V_\beta$, and the formal character of $V$ is given by:
\[
\mathrm{char}(V) = \sum_{\beta \in P} \dim(V_\beta) t^\beta \in \mathcal{O}.
\]
\subsection{Highest weight crystals}
The theory of crystal bases provides a combinatorial model for the representation theory of the quantized universal enveloping algebra $U_q(\mathfrak{g})$ in the limit $q \to 0$. A \emph{(normal) crystal} $\Bcal$ is a finite set equipped with the data of: a weight map $\mathrm{wt} : \Bcal \to P$, raising and lowering operators $e_i, f_i : \Bcal \to \Bcal \sqcup \{ 0 \}$ for each $i \in I$ such that, for each $x,y \in \Bcal$,
\begin{enumerate}
\item[(A1)] $f_i(x) = y$ if and only if $e_i(y) = x$, in which case $\mathrm{wt}(y) = \mathrm{wt}(x) - \alpha_i$,
\end{enumerate}
along with string lengths $\varepsilon_i, \varphi_i : \Bcal \to \ZZ$ satisfying:
\begin{enumerate}
\item[(A2)] $\varphi_i(x) = \mathrm{max}( k : f_i^k(x) \neq 0 \}$, $\varepsilon_i(x) = \mathrm{max}( k : e_i^k(x) \neq 0 \}$, and
\item[(A3)] $\varphi_i(x) = \langle \mathrm{wt}(x), \alpha_i^\vee \rangle + \varepsilon(x)$.
\end{enumerate}
We depict a normal crystal as a directed, weighted graph with directed edges given by the lowering operators $\{ f_i \}_{i \in I}$. For any $i \in I$ and $x \in \Bcal$, let $f_i^*(x) = f_i^{\varphi_i(x)}$ (resp. $e_i^*(x) = e_i^{\varepsilon_i(x)}$), and define an equivalence relation $\sim_i$ on $\Bcal$ by $x \sim_i y$ if and only if $x = e_i^k(y)$, or vice-versa, for some $k \ge 0$. The equivalence classes are called \emph{$i$-strings}, and each equivalence class $S$ contains a unique highest weight element $u_S$ given by $u_S = e_i^*(x)$ for any $x \in S$.

For each $\lambda \in P^+$ the $G$-module $V(\lambda)$ admits a \emph{highest weight crystal} $\Bcal(\lambda)$ with a highest weight element $b_\lambda$ satisfying $\mathrm{wt}(b_\lambda) = \lambda$ and $e_i(b_\lambda) = 0$ for all $i \in I$. From now on we will only consider normal crystals whose connected components are isomorphic to $\Bcal(\lambda)$ for some $\lambda \in P^+$. The \emph{character} of a crystal $\Bcal$ is:
\[
\mathrm{char}(\Bcal) = \sum_{b \in \Bcal} t^{\mathrm{wt}(x)} \in \mathcal{O},
\]
and $\mathrm{char}(\Bcal(\lambda)) = \mathrm{char}(V(\lambda))$ recovers the formal character of the simple $G$-module $V(\lambda)$.

Furthermore, Kashiwara \cite{Kas90} defines the tensor product $\Bcal \otimes \mathcal{C}$ of two normal crystals; the vertex set is the Cartesian product $\Bcal \times \mathcal{C}$, the weight map is given by $\mathrm{wt}(x \otimes y) = \mathrm{wt}(x) + \mathrm{wt}(y)$, and the lowering operators $f_i$ are determined by the rule:
\[
f_i(x \otimes y) = \begin{cases}
f_i(x) \otimes y & \text{if } \varepsilon_i(x) \ge \varphi_i(y), \\
x \otimes f_i(y) & \text{if } \varepsilon_i(x) < \varphi_i(y).
\end{cases}
\]
The tensor product $\Bcal(\lambda) \otimes \Bcal(\mu)$ of highest weight crystals decomposes as a direct sum of highest weight crystals, and this replicates the decomposition on the module level; that is,
\[
V(\lambda) \otimes V(\mu) \cong \bigoplus_{\nu} V(\nu)^{\oplus c_{\lambda \mu}^\nu} \Longleftrightarrow \Bcal(\lambda) \otimes \Bcal(\mu) \cong \bigoplus_{\nu} \Bcal(\nu)^{\oplus c_{\lambda \mu}^\nu}.
\]

\subsection{Tableau crystals}
In type $A_{n-1}$, the dominant weights are identified with \emph{integer partitions} $\lambda = (\lambda_1 \ge \cdots \ge \lambda_n \ge 0) \in (\ZZ_{\ge 0})^n$, and there is a well-known realization of the highest weight crystal $\Bcal(\lambda)$ in terms of \emph{semistandard Young tableaux (SSYT)} of shape $\lambda$ which we will use to construct many of our examples. An SSYT of shape $\lambda$ is a filling $T : \lambda \to \ZZ_{>0}$ of $\lambda$ for which the entries weakly increase across rows, and strictly increase along columns. We denote the set of such fillings by $SSYT(\lambda)$. For instance,
\[
\ytableausetup{boxsize=1em}
\begin{ytableau}
3 & 4 \\
1 & 1 & 2
\end{ytableau} \in SSYT(3,2), \quad \text{but} \quad
\begin{ytableau}
3 & 4 \\
1 & {\color{red} 3} & {\color{red} 2}
\end{ytableau} \not\in SSYT(3,2), \quad
\begin{ytableau}
{\color{red} 1} & 4 \\
{\color{red} 1} & 1 & 2
\end{ytableau} \not\in SSYT(3,2).
\]
Note that we are using French (coordinate) notation to depict our integer partitions. 

The underlying vertex set for the type $A_{n-1}$ highest weight crystal $\Bcal(\lambda)$ is $SSYT_n(\lambda)$, which is the subset of $SSYT(\lambda)$ consisting of those tableaux filled with entries $\le n$. For instance,
\[
\ytableausetup{boxsize=1em}
\begin{ytableau}
2 & 4 \\
1 & 1 & 2
\end{ytableau} \in SSYT_4(3,2,0,0), \quad \text{but} \quad 
\begin{ytableau}
2 & {\color{red} 4} \\
1 & 1 & 2
\end{ytableau} \not\in SSYT_3(3,2,0).
\]
The weight map is given by:
\[
\mathrm{wt}(T) = (T^{-1}(1), \ldots, T^{-1}(n)) \in (\ZZ_{\ge 0})^n,
\]
where $T^{-1}(i)$ denotes the number of entries $=i$ in $T$. For instance, the weight of the above tableau is $(2,2,0,1)$. The raising and lowering operators are defined as follows.
\begin{defn}[\cite{KN94},\cite{Lit}]
Let $\lambda = (\lambda_1 \ge \cdots \ge \lambda_n \ge 0)$ and $T \in SSYT_n(\lambda)$. For $1 \le i \le n-1$, we iteratively \emph{$i$-pair} the entries of $i, i+1$ in $T$ by pairing an $i+1$ with an $i$ lying weakly to its right whenever all entries of $i,i+1$ lying weakly between them are already $i$-paired. Then:
\begin{itemize}
\item if all entries of $i+1$ in $T$ are $i$-paired, then $e_i(T) = 0$; otherwise, $e_i(T)$ is obtained by changing the \emph{leftmost} $i$-unpaired $i+1$ in $T$ into an $i$.
\item if all entries of $i$ in $T$ are $i$-paired, then $f_i(T) = 0$; otherwise, $f_i(T)$ is obtained by changing the \emph{rightmost} $i$-unpaired $i$ in $T$ into an $i+1$.
\end{itemize}
\end{defn}
The data of $(SSYT_n(\lambda), \{e_i, f_i, \varepsilon_i, \varphi_i \}_{i \in I})$ determines a directed, weighted graph which is isomorphic to the highest weight crystal $\Bcal(\lambda)$. There are also tableaux models for the other classical Lie types in \cite{KN94} and \cite{Lit} which we will not define here.

\subsection{Demazure crystals}
Demazure \cite{Dem74a} studied certain cyclic $B$-submodules $V_w(\lambda)$ of $V(\lambda)$. To define them, recall that $V(\lambda)$ decomposes into weight spaces as $V(\lambda) \cong \bigoplus_{\beta \in P} V(\lambda)_\beta$. The \emph{extremal weight spaces} are $\{ V(\lambda)_{w \cdot \lambda} : w \in W \}$, and are each one-dimensional. For any $w \in W$, the \emph{Demazure module} $V_w(\lambda)$ is the $B$-module generated by the extremal weight space $V(\lambda)_{w \cdot \lambda}$, and for $v, w \in W^\lambda$, we have the containment $V_v(\lambda) \subseteq V_w(\lambda)$ if and only if $v \le w$ in Bruhat order.

The character of $V_w(\lambda)$ can be obtained as follows. For any $\beta \in P$ and $i \in I$, the \emph{Demazure operator} $D_i$ acts on $\mathcal{O}$ via:
\[
D_i(t^\beta) = \frac{t^{\beta + \rho} - t^{s_i \cdot (\beta + \rho)}}{1 - t^{-\alpha_i}} t^{-\rho},
\]
where $\rho$ is the Weyl vector. The Demazure operators are idempotent and satisfy the same braid relations as the underlying Weyl group, so we may define $D_w := D_{i_k} \cdots D_{i_1}$ for any $w \in W$ with reduced word $w = s_{i_k} \cdots s_{i_1}$. The \emph{refined Demazure character formula} is:
\begin{thm}[\cite{Jos},\cite{Kas93},\cite{Lit}]
For $\lambda \in P^+$ and $w \in W$,
\[
\mathrm{char}(V_w(\lambda)) = D_w(t^\lambda) \in \mathcal{O}.
\]
\end{thm}

Littelmann \cite{Lit} proved the existence of a subset $\Bcal_w(\lambda)$ whose character agrees with $V_w(\lambda)$; Kashiwara \cite{Kas93} constructed it and derived many of its beautiful properties. By analogy, $\Bcal_w(\lambda)$ is called a \emph{Demazure crystal}. To understand its structure we require the following operators, defined for any $i \in I$ and any $X \subseteq \Bcal(\lambda)$:
\begin{align*}
\Fcal_i(X) &:= \{ f_i^d(x) : x \in X, d \ge 0 \} - \{ 0 \} \subseteq \Bcal(\lambda), \\
\Ecal_i(X) &:= \{ e_i^d(x) : x \in X, d \ge 0 \} - \{ 0 \} \subseteq \Bcal(\lambda).
\end{align*}
\begin{defn}[\cite{Kas93}]
Let $\lambda \in P^+$ and $w \in W$ with $w = s_{i_k} \cdots s_{i_1}$ a reduced word. The \emph{Demazure crystal} $\Bcal_w(\lambda)$ is:
\[
\Bcal_w(\lambda) = \Fcal_{i_k} \cdots \Fcal_{i_1}(\{ b_\lambda \}),
\]
where $b_\lambda$ is the highest weight element of $\Bcal(\lambda)$.
\end{defn}
Although it is not obvious from the definition, the set $\Fcal_{i_k} \cdots \Fcal_{i_1}(\{ b_\lambda \})$ is independent of the choice of reduced word for $w$, so $\Bcal_w(\lambda)$ is well-defined (see, for instance, \cite{Kas93} or \cite{BS} Theorem 13.5). Strictly speaking, $\Fcal_{i_k} \cdots \Fcal_{i_1}(\{ b_\lambda \})$ determines the underlying vertex set, and $\Bcal_w(\lambda)$ is the corresponding induced subgraph of $\Bcal(\lambda)$.

The element $b_{w \cdot \lambda} := f_{i_k}^* \cdots f_{i_1}^*(b_\lambda) \in \Bcal_w(\lambda)$ has $\mathrm{wt}(b_{w \cdot \lambda}) = w \cdot \lambda$, and $b_{w \cdot \lambda}$ is the \emph{unique} element of its weight in $\Bcal(\lambda)$. As for Demazure modules, such elements are called \emph{extremal weight elements}. Note that, if $w \cdot \lambda = w' \cdot \lambda$, then $\Bcal_w(\lambda) = \Bcal_{w'}(\lambda)$, so the Demazure crystal $\Bcal_w(\lambda)$ is determined solely by the left coset of $W / W_\lambda$ containing $w$. Thus we will often assume for simplicity that $w \in W^\lambda$ is a minimal-length coset representative.

\begin{thm}[\cite{Kas93}]\label{thm:dem-def} Demazure crystals satisfy the following properties:
\begin{enumerate}
\item $\Ecal_i(\Bcal_w(\lambda)) \subseteq \Bcal_w(\lambda)$ for all $i \in I$,
\item if $s_iw < w$, then
\[
\Bcal_w(\lambda) = \{ f_i^k(x) : x \in \Bcal_{s_iw}(b_\lambda), k \ge 0, e_i(x) = 0 \} - \{ 0 \},
\]
so that in particular $\Fcal_i(\Bcal_w(\lambda)) = \Bcal_w(\lambda)$, and
\item for any $i$-string $S$ in $\Bcal_w(\lambda)$,
\[
S \cap \Bcal_w(\lambda) \in \{ \emptyset, \{ u_S \}, S \},
\]
where $u_S$ is the highest weight element of $S$.
\end{enumerate}
\end{thm}
Note that we only have $\Fcal_i(\Bcal_w(\lambda)) \subseteq \Bcal_w(\lambda)$ when $s_iw < w$; this may not be the case when $s_iw > w$ (in fact it \emph{never} is, so long as $w,s_iw$ lie in different cosets of $W / W_\lambda$). However, $\Bcal_w(\lambda)$ is closed under $\Ecal_i$ for all $i \in I$. The third property of Theorem \ref{thm:dem-def} is noticeably different from the first two --- we call a subset which satisfies this condition \emph{extremal} --- and it will receive a closer look in Section \ref{sec:extremal}.

One may realize a Demazure crystal of classical type combinatorially as the truncation of a tableau highest weight crystal, but there is also a dedicated tableau model for type $A$ Demazure crystals which uses objects called key tableaux, due to \cite{AS}.

\section{Demazure atoms}\label{sec:atoms}
Just as Demazure modules are naturally filtered by Bruhat order, there is a well-known analogous result for Demazure crystals:
\begin{prop}
For any $\lambda \in P^+$ and $v, w \in W^\lambda$, we have the containment $\Bcal_v(\lambda) \subseteq \Bcal_w(\lambda)$ of Demazure crystals if and only if $v \le w$ in Bruhat order.
\end{prop}
Because of this containment, it is natural to study the minimal nonintersecting pieces of Demazure crystals, which was initiated by Lascoux-Sch\"{u}tzenberger on the level of characters in type $A$. After Mason \cite{Mason}, these pieces are referred to as \emph{Demazure atoms}.
\begin{defn}[\cite{Kas02}]
Let $\lambda \in P^+$ and $w \in W^\lambda$. The \emph{(crystal) Demazure atom} $\mathcal{A}_w(\lambda)$ is
\[
\mathcal{A}_w(\lambda) = \Bcal_w(\lambda) \setminus \bigcup_{v < w} \Bcal_v(\lambda).
\]
\end{defn}
Intuitively, the crystal Demazure atom $\mathcal{A}_w(\lambda)$ consists of those elements of $\Bcal_w(\lambda)$ which were added at the ``last step", i.e. those elements which are not contained in any strictly smaller Demazure crystal. Much of the literature concerning Demazure atoms is on the level of characters; in type $A$, the characters of Demazure crystals and Demazure atoms form combinatorially interesting bases for the polynomial ring $\ZZ[x_1,\ldots,x_n]$, and they were studied extensively in \cite{LS} and \cite{Mason}. Following \cite{BBBG}, we will refer to the character of a crystal Demazure atom as a \emph{polynomial Demazure atom} when there may be confusion.

We now briefly review this theory in type $A_{n-1}$, where $W = S_n$, $P = \ZZ^n / \langle (1,\ldots,1) \rangle$, and $\mathcal{O} = \ZZ[t_1,\ldots,t_n]$. For $1 \le i \le n-1$, the Demazure operator $D_i$ specializes to the \emph{divided difference operator} $\pi_i$, which acts on $f \in \ZZ[t_1,\ldots,t_n]$ via:
\[
\pi_i(f) = \frac{t_if - s_i \cdot (t_if)}{t_i - t_{i+1}} \in \ZZ[t_1,\ldots,t_n],
\]
where $s_i \cdot f$ is obtained by swapping $t_i$ and $t_{i+1}$ in $f$. Lascoux-Sch\"{u}tzenberger define a related operator to obtain polynomial Demazure atoms:
\begin{thm}[\cite{LS}]
For $1 \le i \le n-1$, let $\theta_i$ denote the operator on $\ZZ[t_1,\ldots,t_n]$ given by
\[
\theta_i(f) = \pi_i(f) - f.
\]
Then for $\lambda \in P^+$ and $w \in W^\lambda$ with $w = s_{i_k} \cdots s_{i_1}$ a reduced word, the (polynomial) Demazure atom is:
\[
\mathrm{char}(\mathcal{A}_w(\lambda)) = \theta_{i_k} \cdots \theta_{i_1}(t_1^{\lambda_1} \cdots t_n^{\lambda_n}).
\]
\end{thm}
The Demazure character is often denoted by $\kappa_\beta = \mathrm{char}(\Bcal_w(\lambda))$, where $\beta = w \cdot \lambda$ is a weak composition, and called a \emph{key polynomial}. We can similarly index Demazure atoms by weak compositions in this case, and they are denoted by $\mathfrak{A}_\beta = \mathrm{char}(\mathcal{A}_w(\lambda))$. In fact, in type $A_{n-1}$, key polynomials $\{ \kappa_\beta \}_\beta$ and polynomial Demazure atoms $\{ \mathfrak{A}_\beta \}_\beta$ both form bases for $\ZZ[t_1,\ldots,t_n]$ as $\beta$ ranges over the weak compositions of length $n$.

Lascoux and Sch\"{u}tzenberger \cite{LS} define a map $K_+ : SSYT_n(\lambda) \to S_n$ called the \emph{(right) key map} which associates to any $SSYT$ an element $K_+(T)$ of the Weyl group, and Jacon-Lecouvey \cite{JL} generalize the key map to $K_+ : \Bcal(\lambda) \to W$ for $\Bcal(\lambda)$ the highest weight crystal associated to any symmetrizable Kac-Moody algebra. Santos \cite{Santos} also defines the type $C$ key map with a distinct approach using the symplectic plactic monoid. The shared property of these maps is that they determine membership in a given Demazure atom:
\begin{thm}[\cite{LS},\cite{JL}, \cite{Santos}]\label{thm:dem-into-atoms}
For any $\lambda \in P^+$ and $w \in W^\lambda$, we have
\begin{align*}
\Bcal_w(\lambda) &= \{ b \in \Bcal(\lambda) : K_+(b) \le w \}, \\
\mathcal{A}_w(\lambda) &= \{ b \in \Bcal(\lambda) : K_+(b) = w \}.
\end{align*}
In particular,
\[
\Bcal_w(\lambda) = \bigsqcup_{v \le w, v \in W^\lambda} \mathcal{A}_v(\lambda).
\]
\end{thm}
Taking $w = w_0$, the preceding theorem shows that the set of crystal Demazure atoms $\{ \mathcal{A}_w(\lambda) : w \in W^\lambda \}$ are pairwise disjoint for any $\lambda \in P^+$.

The key map is profoundly useful in that it gives a uniform way to determine, for any given crystal element, in which Demazure atom it resides. In the following section, we approach the problem from the other direction, and provide a method for recursively generating the entire Demazure atom.

\subsection{Atomic operators}
Fix $\lambda \in P^+$ for the remainder of the section. The results of this section hold in arbitrary finite type, unless otherwise indicated.
\begin{defn}
For any $i \in I$ and $X \subseteq \Bcal(\lambda)$, define
\[
\Acal_i(X) = \Fcal_i(X) \setminus X.
\]
We call $\Acal_i$ an \emph{atomic operator}.
\end{defn}

In type $A$, these operators are equivalent to the map $\overline{\pi}_i$ defined on the plactic monoid in \cite{LS}, but we define them for any highest weight crystal of finite type. The $\Acal_i$'s do not, in general, satisfy the braid relations. However, they do in the following special case.

\begin{thm}\label{thm:atomic-op}
For $w \in W^\lambda$ with $w = s_{i_k} \cdots s_{i_1}$ a reduced word, we have
\[
\Acal_{i_k} \cdots \Acal_{i_1}(b_\lambda) = \mathcal{A}_w(\lambda).
\]
In particular, if $w = s_{j_k} \cdots s_{j_1}$ is another reduced word, then
\[
\Acal_{i_k} \cdots \Acal_{i_1}(b_\lambda) = \Acal_{j_k} \cdots \Acal_{j_1}(b_\lambda).
\]
\end{thm}
\begin{proof}
%First, note that by fixing reduced words $w = s_{i_k} \cdots s_{i_1}$ and $u = s_{j_d} \cdots s_{j_1}$ for any $u < w$, we have
%\[
%\widehat{\Fcal}_w = \{ f_{i_k}^{p_k} \cdots f_{i_1}^{p_1}(b_\lambda) : p_i \ge 0 \} - \bigcup_{(j_d, \ldots, j_1) < (i_k, \ldots, i_1)} \{ f_{j_d}^{q_d} \cdots f_{j_1}^{q_1}(b_\lambda) : q_i \ge 0 \}.
%\]
Proceed by induction on $k = \ell(w)$. When $k = 1$, we have $\Acal_i(b_\lambda) = \Fcal_i(b_\lambda) \setminus \{ b_\lambda \} = \Bcal_{s_i}(\lambda) - \Bcal_{id}(\lambda)$, and the result follows.

Now let $k > 1$, and select $i \in I$ such that $s_iw < w$. Then $s_iw \in W^\lambda$ since, if it were not, then $s_iw \cdot \lambda = w \cdot \lambda$, which would imply that $\Bcal_{s_iw}(\lambda) = \Bcal_w(\lambda)$, contradicting the assumption that $w$ is a minimal-length coset representative in $W / W_\lambda$. Since $s_iw < w$, we have a reduced word for $w$ of the form $s_i s_{i_{k-1}} \cdots s_{i_1}$, so that $s_{i_{k-1}} \cdots s_{i_1}$ is a reduced word for $s_iw$. By induction,
\[
\Acal_{i_{k-1}} \cdots \Acal_{i_1}(b_\lambda) = \mathcal{A}_{s_iw}(\lambda) =  \Bcal_{s_iw}(\lambda) \setminus \bigcup_{v < s_iw} \Bcal_v(\lambda).
\]
We then have
\[
\Acal_i \Acal_{i_{k-1}} \cdots \Acal_{i_1}(b_\lambda) = \Fcal_i\left(\mathcal{A}_{s_iw}(\lambda)\right) \setminus \mathcal{A}_{s_iw}(\lambda).
\]
Denote this set by $L$, and first notice that $L \subseteq \Fcal_i(\Bcal_{s_iw}(\lambda)) = \Bcal_w(\lambda)$. We aim to show that $L = \mathcal{A}_w(\lambda)$. Let $x \in L$, so that $x \not\in \mathcal{A}_{s_iw}(\lambda)$, but $e_i^*(x)$ is. Assume to the contrary that $x \not\in \mathcal{A}_w(b_\lambda)$, and pick $v < w$ such that $x \in \Bcal_v(\lambda)$ and $\ell(v)$ is minimized. Then either:
\begin{enumerate}
\item[i.] $v = s_iw$, in which case $x \in \Bcal_{s_iw}(\lambda) \setminus \bigcup_{v < s_iw} \Bcal_v(\lambda) = \mathcal{A}_{s_iw}(\lambda)$, which is a contradiction,
\item[ii.] $v < s_iw$, in which case $e_i^*(x) \in \Bcal_v(\lambda)$ by Theorem \ref{thm:dem-def}(1). But we assumed $e_i^*(x) \in \mathcal{A}_{s_iw}(\lambda)$ and $\mathcal{A}_{s_iw}(\lambda) \cap \Bcal_v(\lambda) = \emptyset$, so this is a contradiction,
\item[iii.] $v \not< s_iw$. We have $v < w$ since $L \subseteq \Bcal_w(b_\lambda)$, so this forces $s_iv < v$ for the following reason: every reduced word for $v$ is a subword of $s_i s_{i_{k-1}} \cdots s_{i_1}$, and by assumption, there is no reduced word for $v$ which is also a subword of $s_{i_{k-1}} \cdots s_{i_1}$. Thus $v$ must have a reduced word ending in $s_i$, i.e. $s_iv < v$. For this reason, we also have $s_iv < s_iw$. But this forces $e_i^*(x) \in \Bcal_{s_iv}(\lambda)$ by Theorem \ref{thm:dem-def}.2, which is again a contradiction since, by assumption, $e_i^*(x) \in \mathcal{A}_{s_iw}(\lambda)$ and $\mathcal{A}_{s_iw}(\lambda) \cap \Bcal_{s_iv}(\lambda) = \emptyset$.
\end{enumerate}
Thus $x \in \mathcal{A}_w(\lambda)$, so that $L \subseteq \mathcal{A}_w(\lambda)$.

Now assume $y \in \mathcal{A}_w(\lambda)$. Then $y \not\in \Bcal_{s_iw}(\lambda)$ since $s_iw < w$, so $y \not\in \mathcal{A}_{s_iw}(\lambda)$. Since $y \in \Bcal_w(\lambda)$ and $s_iw < w$, we have by Theorem \ref{thm:dem-def}(2) that $e_i^*(y) \in \Bcal_{s_iw}(\lambda)$. Now, it suffices to show that $e_i^*(y) \in \mathcal{A}_{s_iw}(\lambda)$, i.e. $e_i^*(y) \not\in \Bcal_v(\lambda)$ for any $v < s_iw$. If this were the case for some $v < s_iw$, then either $s_iv < v$ and $y \in \Fcal_i\Bcal_v(\lambda) \subseteq \Bcal_v(\lambda)$ by Theorem \ref{thm:dem-def}(2), or $v < s_iv$ and $y \in \Bcal_{s_iv}(\lambda)$. We have $v, s_iv < w$ in either case, so this is a contradiction since $y \in \Bcal_w(\lambda) \setminus \bigcup_{v < w} \Bcal_v(\lambda)$.

Thus $y \in L$ and $\mathcal{A}_w(\lambda) \subseteq L$; we conclude finally that the two sets are equal.
\end{proof}
For example, Figure \ref{fig:atom-decomp} shows how the type $A_2$ Demazure crystal $\Bcal_{s_1s_2}(3,2,0)$ decomposes into crystal Demazure atoms, using the atomic operators.
\begin{figure}\label{fig:atom-decomp}
  \begin{center}
  \begin{tikzpicture}
  \ytableausetup{boxsize=1em}
  \draw (0,0) node (320) {{\color{orange}$\begin{ytableau}
  {2} & {2} \\
  {1} & {1} & {1}
  \end{ytableau}$}};
  
  \draw (-1.5,-1.5) node (311) {{\color{darkpurple}$\begin{ytableau}
  {2} & {3} \\
  {1} & {1} & {1}
  \end{ytableau}$}};

  \draw (-3,-3) node (302) {{\color{darkpurple}$\begin{ytableau}
  {3} & {3} \\
  {1} & {1} & {1}
  \end{ytableau}$}};
  
  \draw (2.5,0) node (230) {{\color{teal}$\begin{ytableau}
  {2} & {2} \\
  {1} & {1} & {2}
  \end{ytableau}$}};
  
  \draw (1,-1.5) node (221) {{\color{magenta}$\begin{ytableau}
  {2} & {3} \\
  {1} & {1} & {2}
  \end{ytableau}$}};
  
  \draw (3.5,-1.5) node (131) {{\color{magenta}$\begin{ytableau}
  {2} & {3} \\
  {1} & {2} & {2}
  \end{ytableau}$}};
  
  \draw (-0.5,-3) node (212) {{\color{magenta}$\begin{ytableau}
  {3} & {3} \\
  {1} & {1} & {2}
  \end{ytableau}$}};
  
  \draw (2,-3) node (122) {{\color{magenta}$\begin{ytableau}
  {3} & {3} \\
  {1} & {2} & {2}
  \end{ytableau}$}};
  
  \draw (4.5,-3) node (032) {{\color{magenta}$\begin{ytableau}
  {3} & {3} \\
  {2} & {2} & {2}
  \end{ytableau}$}};
  
  \draw[->,color=darkblue] (320) -- (311) node [midway,above,darkblue] {${f_2}$};
  \draw[->,color=darkblue] (311) -- (302) [midway,above,darkblue];
  
  \draw[->,color=darkred] (320) -- (230) node [midway,above,darkred] {${f_1}$};
  
  \draw[->,color=darkgreen] (311) -- (221) [midway,above,darkred];
  \draw[->,color=darkgreen] (221) -- (131) [midway,above,darkred];
  
  \draw[->,color=darkred] (221) -- (212) [midway,above,darkblue];
  
  \draw[->,color=darkgreen] (302) -- (212) [midway,above,darkred];
  \draw[->,color=darkgreen] (212) -- (122) [midway,above,darkred];
  \draw[->,color=darkgreen] (122) -- (032) [midway,above,darkred];
  
  \draw[dashed] (2,0.8) -- (-2.5,-3.8);
  \draw[dashed] (-4,-0.7) -- (7,-0.7);
  
  \draw (-2,0) node {${\color{orange}\Acal_{id}(b_{320})}$};
  \draw (-4,-2) node {${\color{darkpurple}\Acal_{2}(b_{320})}$};
  \draw (4.25,0) node {${\color{teal}\Acal_{1}(b_{320})}$};
  \draw (6,-2) node {${\color{magenta}\Acal_{1}\Acal_{2}(b_{320})}$};
  
  %\node[ellipse, draw=red, fit=(311) (302), inner sep=-1mm] (all) {};
  \end{tikzpicture}
  \end{center}
  \caption{$\Bcal_{s_1s_2}(3,2,0) = {\color{orange}\Acal_{id}(b_{320})} \sqcup {\color{darkpurple}\Acal_{2}(b_{320})} \sqcup {\color{teal}\Acal_{1}(b_{320})} \sqcup {\color{magenta}\Acal_{1}\Acal_{2}(b_{320})}$}
\end{figure}

\begin{cor}\label{cor:recursive}
For $w \in W^\lambda,$ and $i \in I$ such that $s_iw < w$, we have
\[
\mathcal{A}_w(\lambda) = \Acal_i(\mathcal{A}_{s_iw}(\lambda)).
\]
In particular, for each $x \in \mathcal{A}_w(\lambda)$ and any reduced word $w = s_{i_k} \cdots s_{i_1}$, there exist $d_1, \ldots, d_k > 0$ such that:
\[
x = f_{i_k}^{d_k} \cdots f_{i_1}^{d_1}(b_\lambda),
\]
and for each $1 \le j \le k-1$,
\begin{equation}\label{eq:intermediate}
e_{i_j}(f_{i_{j-1}}^{d_{j-1}} \cdots f_{i_1}^{d_1}) = 0.
\end{equation}
\end{cor}
\begin{proof}
For the first assertion, note that if $s_iw < w$, then $s_iw \in W^\lambda$ and there exists a reduced word for $w$ of the form $w = s_is_{i_{k-1}} \cdots s_{i_1}$, so that $s_iw = s_{i_{k-1}} \cdots s_{i_1}$. Thus by Theorem \ref{thm:atomic-op},
\[
\mathcal{A}_w(\lambda) = \Acal_i\Acal_{i_{i-1}} \cdots \Acal_{i_1}(b_\lambda) = \Acal_i(\mathcal{A}_{s_iw}(\lambda)).
\]
The second assertion now follows from the first by induction on $\ell(w)$. The base case is trivial, and if $\ell(w) > 1$ with reduced word $w = s_is_{i_{k-1}} \cdots s_{i_1}$, then we have $s_iw < w$. Thus by the first assertion, for any $x \in \mathcal{A}_w(\lambda)$, we have $e_i(x) \neq 0$, and $e_i^*(x) =: x' \in \mathcal{A}_{s_iw}(\lambda)$. By induction $x' = f_{i_{k-1}}^{d_{k-1}} \cdots f_{i_1}^{d_1}(b_\lambda)$ for $d_1, \ldots, d_{k-1} > 0$ satisfying Equation \eqref{eq:intermediate}, and $x = f_i^d(x')$ for some $d > 0$, so
\[
x = f_i^df_{i_{k-1}}^{d_{k-1}} \cdots f_{i_1}^{d_1}(b_\lambda),
\]
and since $e_i(x') = 0$, $x$ satisfies Equation \eqref{eq:intermediate}.
\end{proof}
The atomic operators therefore provide a recursive way of obtaining entire crystal Demazure atoms. It is in fact necessary to require $w \in W^\lambda$ in the above theorem, as illustrated by the following proposition.
\begin{prop}
For $w \in W^\lambda$, we have
\[
\Fcal_i(\mathcal{A}_w(\lambda)) = \begin{cases}
\mathcal{A}_w(\lambda) & \text{if } s_iw < w \text{ or } s_iw \not\in W^\lambda, \\
\mathcal{A}_w(\lambda) \sqcup \mathcal{A}_{s_iw}(\lambda) & \text{if } w < s_iw \text{ and } w \in W^\lambda.
\end{cases}
\]
\end{prop}
\begin{proof}
First assume $s_iw < w$, so that $\Fcal_i(\Bcal_w(\lambda)) = \Bcal_w(\lambda)$ by Theorem \ref{thm:dem-def}(2). Then $\Fcal_i(\mathcal{A}_w(\lambda)) \subseteq \Bcal_w(\lambda)$, and if there exists $x \in \mathcal{A}_w(\lambda)$ and $v < w$ such that $f_i^d(x) \in \Bcal_v(\lambda)$ for some $d > 0$, then $x \in \Bcal_v(\lambda)$ since $\Bcal_v(\lambda)$ satisfies Theorem \ref{thm:dem-def}(3). But this is a contradiction, so $f_i^d(x) \in \mathcal{A}_w(\lambda)$ for all $x \in \mathcal{A}_w(\lambda)$, i.e. $\Fcal_i(\mathcal{A}_w(\lambda)) = \mathcal{A}_w(\lambda)$. If $w < s_iw$ and $s_iw \not\in W^\lambda$ then $w \cdot \lambda = s_iw \cdot \lambda$, so that $\Fcal_i(\Bcal_w(\lambda)) = \Bcal_w(\lambda)$. Thus it follows from the above argument that $\Fcal_i(\mathcal{A}_w(\lambda)) = \mathcal{A}_w(\lambda)$.

Now assume $w < s_iw$, and that $s_iw \in W^\lambda$. Then by definition we have
\[
\mathcal{A}_{s_iw}(\lambda) = \Acal_i(\mathcal{A}_w(\lambda)) = \Fcal_i(\mathcal{A}_w(\lambda)) \setminus \mathcal{A}_w(\lambda),
\]
or $\Fcal_i(\mathcal{A}_w(\lambda)) = \mathcal{A}_w(\lambda) \sqcup \mathcal{A}_{s_iw}(\lambda)$.
\end{proof}
This proposition shows that, if we pick $w \not\in W^\lambda$ and a reduced word $w = s_{i_k} \cdots s_{i_1}$, then letting $j$ be minimal such that $s_{i_j} \cdots s_{i_1} \in W^\lambda$ but $s_{i_{j+1}} s_{i_j} \cdots s_{i_1} \not\in W^\lambda$, then
\[
\Acal_{i_{j+1}}(\mathcal{A}_{s_{i_j} \cdots s_{i_1}}(\lambda)) = \Fcal_{i_{j+1}}(\mathcal{A}_{s_{i_j} \cdots s_{i_1}}(\lambda)) \setminus \mathcal{A}_{s_{i_j} \cdots s_{i_1}}(\lambda) = \mathcal{A}_{s_{i_j} \cdots s_{i_1}}(\lambda) \setminus \mathcal{A}_{s_{i_j} \cdots s_{i_1}}(\lambda) = \emptyset,
\]
so that $\Acal_{i_k} \cdots \Acal_{i_1}(b_\lambda) = \emptyset$.

\subsection{Atom-positivity}
As seen above, Demazure crystals decompose as a disjoint union of Demazure atoms, and in particular their characters are positive sums of polynomial Demazure atoms. We generalize this notion as follows:
\begin{defn}
A subset $X \subseteq \Bcal(\lambda)$ is called \emph{strongly atom-positive} if there exists $L_X \subseteq W^\lambda$ such that
\[
X = \bigsqcup_{w \in L_X} \mathcal{A}_w(\lambda),
\]
and $X \subseteq \Bcal(\lambda)$ is \emph{weakly atom-positive} if there exist $\lambda^1,\ldots,\lambda^k \in P^+$ and $w^1, \ldots, w^k \in W$ with each $w^i \in W^{\lambda^i}$ such that
\[
\mathrm{char}(X) = \sum_{i=1}^k \mathrm{char}(\mathcal{A}_{w^i}(\lambda^i)).
\]
%A subset $\mathcal{X} \subseteq \Bcal$ of an arbitrary normal crystal will be called strongly (resp. weakly) atom-positive if each of its connected components is strongly (resp. weakly) atom-positive.
\end{defn}
As the names suggest, strong atom-positivity implies weak atom-positivity, but not conversely:
\begin{example}\label{rem:atom}
The following subset of the type $A_2$ crystal $\Bcal(3,2,0)$:
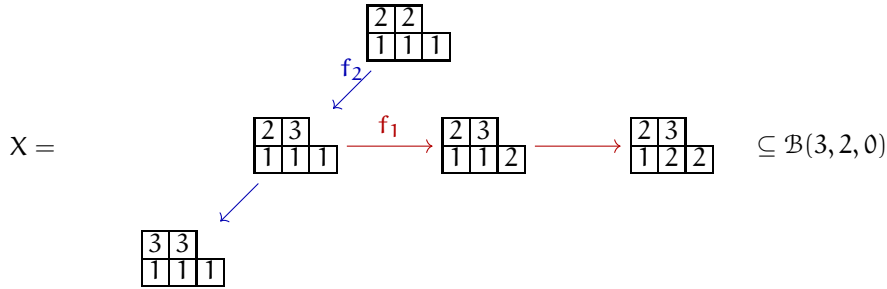
\[
\begin{tikzpicture}
\draw (-5,-1.5) node {$X=$};
\draw (0,0) node (u) {$\begin{ytableau}
2 & 2 \\
1 & 1 & 1
\end{ytableau}$};
\draw (-1.5,-1.5) node (u2) {$\begin{ytableau}
2 & 3 \\
1 & 1 & 1
\end{ytableau}$};
\draw (1,-1.5) node (u21) {$\begin{ytableau}
2 & 3 \\
1 & 1 & 2
\end{ytableau}$};
\draw (3.5,-1.5) node (u211) {$\begin{ytableau}
2 & 3 \\
1 & 2 & 2
\end{ytableau}$};
\draw (-3,-3) node (u22) {$\begin{ytableau}
3 & 3 \\
1 & 1 & 1
\end{ytableau}$};

\draw[->] (u) edge[darkblue,midway,above] node {$f_2$} (u2);
\draw[->] (u2) edge[darkblue,midway,above] (u22);

\draw[->] (u2) edge[darkred,midway,above] node {$f_1$} (u21);
\draw[->] (u21) edge[darkred,midway,above] (u211);

\draw (5.5,-1.5) node {$\subseteq \Bcal(3,2,0)$};
\end{tikzpicture}
\]
satisfies $\mathrm{char}(X) = \mathfrak{A}_{3,2,0)}+ \mathfrak{A}_{(3,0,2)} + \mathfrak{A}_{(1,3,1)}$. However, one can check that $X$ is \emph{not} a disjoint union of crystal Demazure atoms, i.e. $X$ is weakly atom-positive but not strongly atom-positive.
\end{example}
Arbitrary unions of Demazure crystals are strongly atom-positive. The following theorem follows, more or less, from the corresponding decomposition for Demazure crystals.
\begin{thm}\label{thm:ext-union}
For $w^1, \ldots, w^k \in W^\lambda$, we have:
\[
\bigcup_{i=1}^k \Bcal_{w^i}(\lambda) = \bigsqcup_{w \le w^i \text{ for some } i \text{ and } w \in W^\lambda} \mathcal{A}_w(b_\lambda).
\]
\end{thm}
\begin{proof}
Let $\mathcal{D} = \bigcup_{i=1}^k \Bcal_{w^i}(\lambda)$. For any $w \in W^\lambda$, we will verify that
\[
\mathcal{D} \cap \mathcal{A}_w(\lambda) = \begin{cases}
\mathcal{A}_w(\lambda) & \text{if } w \le w^i \text{ for some } i, \\
\emptyset & \text{otherwise.}
\end{cases}
\]
First, if $w \le w^i$ for some $1 \le i \le k$, then by Theorem \ref{thm:dem-into-atoms} we have $\mathcal{A}_w(\lambda) \subseteq \Bcal_{w^i}(\lambda) \subseteq \mathcal{D}$, so $\mathcal{D} \cap \mathcal{A}_w(\lambda) = \mathcal{A}_w(\lambda)$.

If $w \not\le w^i$, then $\Bcal_{w^i}(\lambda) \cap \mathcal{A}_w(\lambda) = \emptyset$, so that if $w \not\le w^i$ for each $1 \le i \le k$, we have $\mathcal{D} \cap \mathcal{A}_w(\lambda) = \emptyset$. Therefore,
\[
\mathcal{D} = \mathcal{D} \cap \Bcal(\lambda) = \mathcal{D} \cap \bigsqcup_{w \in W^\lambda} \mathcal{A}_w(\lambda) = \bigsqcup_{w \le w^{(i)} \text{ for some } i \text{ and } w \in W^\lambda} \mathcal{A}_w(b_\lambda).
\]
\end{proof}
The indexing set in the above theorem forms a \emph{lower order ideal} in Bruhat order; that is, a set $\Ical \subseteq W$ such that $w \in \Ical, v \le w \Rightarrow v \in \Ical$. Any lower order ideal in Bruhat order is thus determined by its maximal elements. If $w^1, \ldots, w^k \in \Ical$ are the maximal elements, then we will write $\Ical = \langle w^1, \ldots, w^k \rangle$. With this notation, we have the following:
\begin{defn}
For $w^1, \ldots, w^k \in W^\lambda$, let $\Ical = \langle w^1, \ldots, w^k \rangle$. The \emph{Schubert crystal} $\Bcal_\Ical(\lambda) \subseteq \Bcal(\lambda)$ is
\[
\Bcal_\Ical(\lambda) = \bigcup_{i=1}^k \Bcal_{w^i}(\lambda),
\]
and the \emph{Schubert character} $\kappa_{\lambda,\Ical} \in \mathcal{O}$ is
\[
\kappa_{\lambda,\Ical} = \mathrm{char}(\Bcal_\Ical(\lambda)).
\]
\end{defn}
Schubert crystals and characters play a significant role in Polo's conjecture, as discussed below.
%A partial converse is given as follows:
%\begin{thm}
%If $X \subseteq \Bcal(\lambda)$ is atom-positive and $\Ecal_i(X) \subseteq X$ for all $i \in I$, then $X$ is a union of Demazure crystals.
%\end{thm}
%\begin{proof}
%Since $X$ is atom-positive, there exists $L_X \subseteq W^\lambda$ such that $X = \bigsqcup_{w \in L_X} \mathcal{A}_{w}(\lambda)$. Let $w^1, \ldots, w^k \in L_X$ denote the maximal elements of $L_X$ with respect to Bruhat order. We will show that
%\[
%X = \bigcup_{i=1}^k \Bcal_{w^i}(\lambda).
%\]
%For any $w \in W$, recall that $b_{w \cdot \lambda} \in \Bcal(\lambda)$ is the unique element of weight $w \cdot \lambda$ in $\Bcal(\lambda)$, and $b_{w \cdot \lambda} \in \mathcal{A}_w(\lambda)$. In particular, for any $1 \le i \le k$, we have $b_{w^i \cdot \lambda} \in X$ since $\mathcal{A}_{w^i}(\lambda) \subseteq X$.
%
%For any such $i$, fix a reduced word $w^i = s_{i_k} \cdots s_{i_1}$. Then $b_{w^i \cdot \lambda} = f_{i_k}^* \cdots f_{i_1}^*(b_\lambda)$. Since $\Ecal_i(X) \subseteq X$, we 
%\end{proof}

\section{Extremal subsets}\label{sec:extremal}
Polo's conjecture \cite{Polo} concerns the structure of the tensor product of Demazure modules $V_v(\lambda) \otimes V_w(\mu)$. To introduce it, we first briefly discuss filtrations of $B$-modules. The Borel-Weil-Bott theorem realizes the simple $G$-module $V(\lambda)$ as the dual of the space of sections $H^0(G/B, \mathcal{L}_{-\lambda})$ of a line bundle $\mathcal{L}_{-\lambda}$ over the flag variety $G/B$. Generalizing this, the Demazure module $V_w(\lambda)$ may be identified with $H^0(X_w, \mathcal{L}_{-\lambda})$, where $X_w = \overline{BwB}/B$ is a \emph{Schubert variety}. As a further generalization, Polo defines the $B$-module $H^0(S,\mathcal{L}_{-\lambda})$, where $S$ is a \emph{union} of Schubert varieties, and calls it a \emph{Schubert module}.

We say that a $B$-module $V$ admits an \emph{excellent filtration} if each successive quotient in the filtration is isomorphic to a Demazure module. The tensor product $V_v(\lambda) \otimes V_w(\lambda)$ does not always admit an excellent filtration; however, Polo conjectured that there should always exist a \emph{Schubert filtration}:
\begin{conj}[\cite{Polo}]
For $\lambda, \mu \in P^+$ and $v, w \in W$, the tensor product $V_v(\lambda) \otimes V_w(\lambda)$ admits a filtration in which each successive quotient is isomorphic to a Schubert module. In particular, there exist $\nu^1, \ldots, \nu^k \in P^+$ and lower order ideals $\Ical^1, \ldots, \Ical^k \subseteq W$ such that
\[
\mathrm{char}(\Bcal_v(\lambda)) \cdot \mathrm{char}(\Bcal_w(\mu)) = \sum_{i=1}^k c_i \kappa_{\nu^i,\Ical^i}
\]
for some $c_i \in \ZZ_{>0}$.
\end{conj}
Polo's conjecture suggests that the tensor product $\Bcal_v(\lambda) \otimes \Bcal_w(\mu)$ should admit a decomposition such that each connected component is a Schubert crystal. However, while Kashiwara's tensor product rule witnesses good filtrations in the category of highest weight crystals --- that is, $\Bcal(\lambda) \otimes \Bcal(\mu)$ decomposes as a direct sum of highest weight crystals, replicating the decomposition on the module level --- this rule is not well-equipped to deal with tensor products of Demazure crystals. That is, when one restricts the tensor product to the subset $\Bcal_v(\lambda) \otimes \Bcal_w(\mu)$, the connected components are not always isomorphic to Schubert crystals. However, they are in the following case:
\begin{thm}[\cite{ADG}]
The Kashiwara tensor product of Demazure crystals decomposes as a direct sum of Demazure crystals if and only if the tensor product satisfies Theorem \ref{thm:dem-def}(3).
\end{thm}
This is in fact a special case of Polo's conjecture, as the connected components of the tensor product are single Demazure crystals, as opposed to Schubert crystals. Kouno first \cite{Kouno} found a global criterion that gives necessary and sufficient conditions for this decomposition to hold, but we opt to work with the local criterion of \cite{ADG}.

\begin{defn}[\cite{AG}]
Any subset $X \subseteq \Bcal(\lambda)$ satisfying Theorem \ref{thm:dem-def}(3) is called an \emph{extremal subset}.
%and a subset $\mathcal{X} \subseteq \Bcal$ of an arbitrary normal crystal will be called extremal if each of its connected components is.
\end{defn}

Note that any extremal subset $X \subseteq \Bcal(\lambda)$ also necessarily satisfies Theorem \ref{thm:dem-def}(1), and thus contains the highest weight element $b_\lambda$. In particular, every extremal subset of a highest weight crystal is connected. The class of extremal subsets is also closed under union and intersection:

\begin{lemma}
If $X, Y \subseteq \Bcal(\lambda)$ are extremal subsets, then so are $X \cup Y$ and $X \cap Y$.
\end{lemma}
\begin{proof}
For a given $i$-string $S$ contained in $\Bcal(\lambda)$, we have $(X \cup Y) \cap S = (X \cap S) \cup (Y \cap S)$, so since $X \cap S = \emptyset, S$, or $u_S$ (resp. $Y \cap S$), it follows that the same is true of $(X \cup Y) \cap S$. A similar argument shows the same result for $(X \cap Y) \cap S$.
\end{proof}

\begin{cor}
Schubert crystals are extremal.
\end{cor}

Thus extremal subsets of highest weight crystals take on particular importance in the context of Polo's conjecture. As such, the original impetus for this work was to study the relationship between extremality and atom-positivity. The relationship, it turns out, is not so direct, as the class of extremal subsets is much broader than the special cases considered above. In particular, there exist extremal subsets which are not weakly atom-positive.
\begin{example}
The highlighted extremal subset $X \subseteq \Bcal(3,2,0,0)$ in Figure \ref{fig:ext-not-atom} satisfies
\[
\mathrm{char}(X) = \mathfrak{A}_{(3,2,0,0)} + \mathfrak{A}_{(3,0,2,0)} + \mathfrak{A}_{(0,3,2,0)} + \mathfrak{A}_{(3,0,0,2)} + \mathfrak{A}_{(2,3,0,2)} - \mathfrak{A}_{(1,3,0,1)}.
\]
\end{example}
This incongruence essentially follows from the fact that crystal Demazure atoms are not, in general, connected, so that an extremal subset may contain a strict subset of a crystal Demazure atom. Conversely, there also exist atom-positive subsets which are not extremal.
\begin{example}
Let $\lambda \in P^+$ such that there exist $i,j,k \in I$ which are pairwise disjoint, satisfying $s_is_k = s_ks_i$ and $s_j, s_is_j, s_ks_is_j \in W^\lambda$. Then, define
\[
X = \mathcal{A}_{id}(\lambda) \sqcup \mathcal{A}_{s_j}(\lambda) \sqcup \mathcal{A}_{s_is_j}(\lambda) \sqcup \mathcal{A}_{s_ks_is_j}(\lambda) \subseteq \Bcal(\lambda),
\]
noting that each set on the right-hand side is nonempty by choice of $i,j,k$. The subset $X$ is furthermore connected by Corollary \ref{cor:recursive}, since in particular there exists a path in $X$ to the highest weight element $b_\lambda$ for any $x \in X$. Now, let $w = s_ks_is_j$. Since $s_is_k = s_ks_i$, we have $s_iw < w$, so by Corollary \ref{cor:recursive} we have $\mathcal{A}_w(\lambda) = \Acal_{i}(\mathcal{A}_{s_ks_j}(\lambda))$; in particular, $e_i^*(x) \in \mathcal{A}_{s_ks_j}(\lambda)$ for any $x \in \mathcal{A}_w(\lambda)$. But $\mathcal{A}_{s_ks_j}(\lambda) \cap X = \emptyset$, i.e. $x \in X$ but $e_i^*(x) \not\in X$. Thus $X$ is a connected, strongly atom-positive subset which is not extremal.
\end{example}
\begin{figure}\label{fig:ext-not-atom}
\begin{center}
\ytableausetup{boxsize=1em}
\begin{tikzpicture}
\draw (0,0) node (u) {$\begin{ytableau}
2 & 2 \\
1 & 1 & 1
\end{ytableau}$};
\draw [opacity=0.3] (3,0) node (u1) {$\begin{ytableau}
2 & 2 \\
1 & 1 & 2
\end{ytableau}$};
\draw (-2,-2) node (u2) {$\begin{ytableau}
2 & 3 \\
1 & 1 & 1
\end{ytableau}$};
\draw (1,-2) node (u21) {$\begin{ytableau}
2 & 3 \\
1 & 1 & 2
\end{ytableau}$};
\draw (4,-2) node (u211) {$\begin{ytableau}
2 & 3 \\
1 & 2 & 2
\end{ytableau}$};
\draw (-4,-4) node (u22) {$\begin{ytableau}
3 & 3 \\
1 & 1 & 1
\end{ytableau}$};
\draw (-1,-4) node (u221) {$\begin{ytableau}
3 & 3 \\
1 & 1 & 2
\end{ytableau}$};
\draw (2,-4) node (u2211) {$\begin{ytableau}
3 & 3 \\
1 & 2 & 2
\end{ytableau}$};
\draw (5,-4) node (u22111) {$\begin{ytableau}
3 & 3 \\
2 & 2 & 2
\end{ytableau}$};

\draw (-2,-5) node (u23) {$\begin{ytableau}
2 & 4 \\
1 & 1 & 1
\end{ytableau}$};
\draw [opacity=0.3] (1,-5) node (u231) {$\begin{ytableau}
2 & 4 \\
1 & 1 & 2
\end{ytableau}$};
\draw [opacity=0.3] (4,-5) node (u2311) {$\begin{ytableau}
2 & 4 \\
1 & 2 & 2
\end{ytableau}$};

\draw (-4,-7) node (u223) {$\begin{ytableau}
3 & 4 \\
1 & 1 & 1
\end{ytableau}$};
\draw (-4,-10) node (u2233) {$\begin{ytableau}
4 & 4 \\
1 & 1 & 1
\end{ytableau}$};

\draw (-1,-7) node (u2231) {$\begin{ytableau}
3 & 4 \\
1 & 1 & 2
\end{ytableau}$};
\draw (2,-7) node (u22311) {$\begin{ytableau}
3 & 4 \\
1 & 2 & 2
\end{ytableau}$};
\draw (5,-7) node (u223111) {$\begin{ytableau}
3 & 4 \\
2 & 2 & 2
\end{ytableau}$};

\draw (-1,-10) node (u22331) {$\begin{ytableau}
4 & 4 \\
1 & 1 & 2
\end{ytableau}$};
\draw (2,-10) node (u223311) {$\begin{ytableau}
4 & 4 \\
1 & 2 & 2
\end{ytableau}$};
\draw (5,-10) node (u2233111) {$\begin{ytableau}
4 & 4 \\
2 & 2 & 2
\end{ytableau}$};

\path[->] (u) edge[darkblue,midway,left] node {$f_2$} (u2);
\path[->] (u2) edge[darkblue,midway,left] (u22);
\path[->] (u21) edge[darkblue,midway,left] (u221);

\path[->] (u) edge[darkred,midway,above] node {$f_1$} (u1);
\path[->] (u2) edge[darkred,midway,above] (u21);
\path[->] (u21) edge[darkred,midway,above] (u211);
\path[->] (u22) edge[darkred,midway,above] (u221);
\path[->] (u221) edge[darkred,midway,above] (u2211);
\path[->] (u2211) edge[darkred,midway,above] (u22111);
\path[->] (u23) edge[darkred,midway,above] (u231);
\path[->] (u231) edge[darkred,midway,above] (u2311);

\path[->] (u22) edge[darkgreen,midway,left] node {$f_3$} (u223);
\path[->] (u2) edge[darkgreen,midway,left] (u23);
\path[->] (u223) edge[darkgreen,midway,left] (u2233);

\path[->] (u21) edge[darkgreen,midway,left] (u231);
\path[->] (u211) edge[darkgreen,midway,left] (u2311);

\path[->] (u223) edge[darkred,midway,above] (u2231);
\path[->] (u2231) edge[darkred,midway,above] (u22311);
\path[->] (u22311) edge[darkred,midway,above] (u223111);

\path[->] (u2233) edge[darkred,midway,above] (u22331);
\path[->] (u22331) edge[darkred,midway,above] (u223311);
\path[->] (u223311) edge[darkred,midway,above] (u2233111);

\path[->] (u23) edge[darkblue,midway,left] (u223);
\path[->] (u221) edge[darkgreen,midway,left] (u2231);
\path[->] (u2231) edge[darkgreen,midway,left] (u22331);
\path[->] (u2211) edge[darkgreen,midway,left] (u22311);
\path[->] (u22311) edge[darkgreen,midway,left] (u223311);
\path[->] (u22111) edge[darkgreen,midway,left] (u223111);
\path[->] (u223111) edge[darkgreen,midway,left] (u2233111);
\end{tikzpicture}
\end{center}
\caption{The type $A_3$ Demazure crystal $\Bcal_{s_1s_3s_2}(3,2,0,0)$, with an extremal subset $X$ highlighted. Note that the crystal Demazure atom $\mathcal{A}_{s_1s_3s_2}(3,2,0,0)$ is not connected.}
\end{figure}
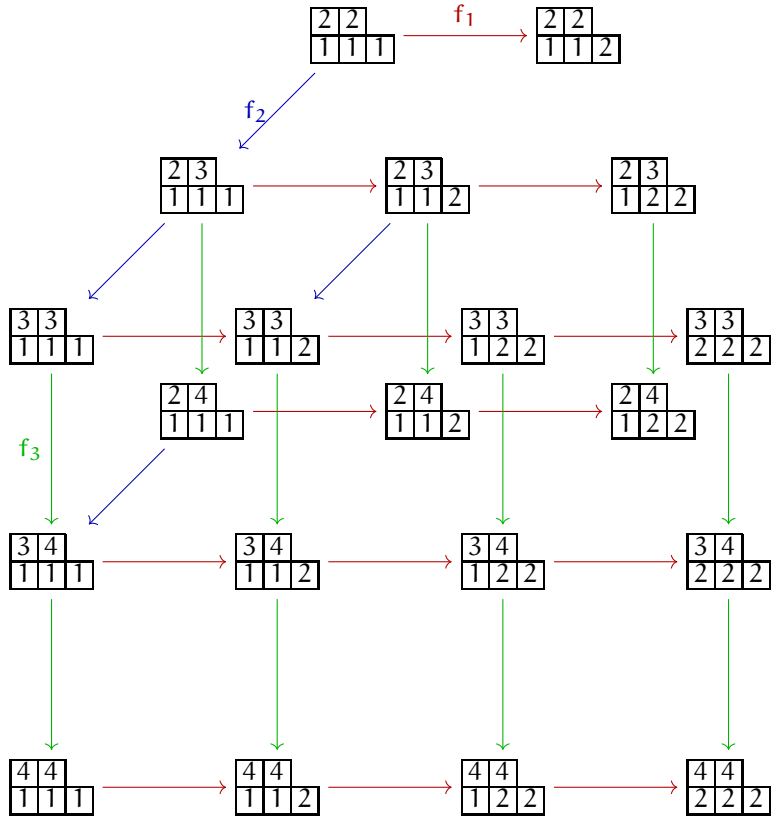

Nevertheless, extremality is still a rather rigid condition to impose on a subset of a highest weight crystal, and we wish to further study the structure of extremal subsets. To this end, Assaf and Gonz\'{a}lez define the \emph{lowest weight elements} of an extremal subset:
\begin{defn}[\cite{AG}]
Given an extremal subset $X \subseteq \Bcal(\lambda)$, an element $x \in X$ is said to be a \emph{lowest weight element} if for each $i \in I$, either $f_i(x) = 0$ or $f_i(x) \not\in X$.
\end{defn}
Lowest weight elements are closely related to extremal weight elements, but are more general: there exist lowest weight elements whose weights are not extremal. For instance, one such element is $f_1^2f_2(b_\lambda)$ in Figure \ref{fig:ext-not-atom}. Note furthermore that if $x$ is a lowest weight element and $0 \neq f_i(x) \not\in X$, then we must have $e_i(x) = 0$ since $X$ is extremal. Thus for each $i \in I$, we have either $f_i(x) = 0$ or $e_i(x) = 0$ (or both), so $x$ must lie at the top or the bottom of each $i$-string passing through $x$. However, there exist crystal elements $x \in \Bcal(\lambda)$ satisfying this weaker condition which are \emph{not} lowest weight elements in any extremal subset $X \subseteq \Bcal(\lambda)$:
\begin{example}
Let $\lambda = (4,4,3,2,0,0)$ and consider the type $A_5$ crystal $\Bcal(\lambda)$, along with the element
\[
\ytableausetup{boxsize=1em,centertableaux}
x = \begin{ytableau}
5 & 6 \\
3 & 4 & 5 \\
2 & 3 & 4 & 4 \\
1 & 2 & 3 & 3
\end{ytableau} \in \Bcal(\lambda).
\]
Then $f_1(x) = f_2(x) = e_3(x) = e_4(x) = e_5(x) = 0$, so $x$ lies either at the top or the bottom of each $i$-string incident to $x$. If $X \subseteq \Bcal(\lambda)$ is an extremal subset containing $x$, then since $e_2(x) \neq 0$, we have
\[
e_2(x) = \begin{ytableau}
5 & 6 \\
3 & 4 & 5 \\
2 & 3 & 4 & 4 \\
1 & 2 & {\color{darkred} 2} & 3
\end{ytableau} \in X.
\]
We then have $e_3e_2(x) \neq 0$, so the $3$-string through $e_2(x)$ must be contained in $X$. In particular,
\[
f_3e_2(x) = \begin{ytableau}
5 & 6 \\
{\color{darkred} 4} & 4 & 5 \\
2 & 3 & 4 & 4 \\
1 & 2 & 2 & 3
\end{ytableau} \in X.
\]
Finally, since $e_2f_3e_2(x) \neq 0$, the $2$-string through $f_3e_2(x)$ must also be contained in $X$, so that
\[
f_2f_3e_2(x) = \begin{ytableau}
5 & 6 \\
4 & 4 & 5 \\
2 & 3 & 4 & 4 \\
1 & 2 & {\color{darkred} 3} & 3
\end{ytableau} \in X.
\]
But $f_2f_3e_2(x) = f_3(x) \in X$, so $x$ is not a lowest weight element in any extremal subset of $\Bcal(\lambda)$.
\end{example}
Extending the $\Ecal_i$ operator, for any $X \subseteq \Bcal(\lambda)$ we define 
\[
\Ecal(X) = \{ e_{i_k}^{d_k} \cdots e_{i_1}^{d_1}(x) : x \in X, i_j \in I, d_j \ge 0 \} - \{ 0 \} \subseteq \Bcal(\lambda).
\]
An extremal subset is generated by its lowest weight elements in the following sense:
\begin{prop}
Let $X \subseteq \Bcal(\lambda)$ be an extremal subset and $\{ z_1, \ldots, z_k \}$ its set of lowest weight elements. Then $X = \Ecal(\{ z_1, \ldots, z_k \})$.
\end{prop}
\begin{proof}
Let $x \in X$. We must show that $x \in \Ecal(z_j)$ for some $1 \le j \le k$, or equivalently, that $z_j \in \Fcal(x)$. If for each $i \in I$ we have $f_i(x) = 0$ or $f_i(x) \not\in X$, then $x$ is a lowest weight element of $X$, and we are done. Otherwise, select $i_1 \in I$ such that $f_{i_1}(x) \in X$. Then $x' := f_{i_1}^*(x) \in X$, since $X$ is extremal. Repeating the above recipe, we find $i_2 \in I$ such that $f_{i_2}(x') \in X$. If this is not possible, then $x' \in \{ z_1, \ldots, z_k \}$. Otherwise, repeat with $x'' := f_{i_2}^*(x')$. Since $X$ is finite, we can continue in this way to obtain $i_1, \ldots, i_k \in I$ such that $f_{i_k}^* \cdots f_{i_1}^*(x) \in \{ z_1, \ldots, z_k \}$, i.e. $x \in \Ecal(\{ z_1, \ldots, z_k \})$.
\end{proof}
As a shorthand, we will say that a subset $X \subseteq \Bcal(\lambda)$ is \emph{$\Ecal$-generated} by $L \subseteq X$ if $X = \Ecal(L)$. While a lowest weight element need not, in general, have extremal weight, it must in the following special case:
\begin{prop}\label{prop:e-gen}\label{prop:ext-char}
If $X \subseteq \Bcal(\lambda)$ is extremal and $X = \Ecal(\{ x \})$ for some $x \in X$, then $x = b_{w \cdot \lambda}$ for some $w \in W^\lambda$.
\end{prop}
\begin{proof}
We will show that there exist $i_1, \ldots, i_k \in I$ such that $x = f_{i_k}^* \cdots f_{i_1}^*(b_\lambda)$. Assume this is not the case, so that for every expression $x =  f_{i_k}^{d_k} \cdots f_{i_1}^{d_1}(b_\lambda)$, there exists an index $j$ such that $d_j < \varphi_{i_j}(f_{i_{j-1}}^{d_{j-1}} \cdots f_{i_1}^{d_1}(b_\lambda))$. Assume $j$ is the minimal such index, so that
\[
x = f_{i_k}^{d_k} \cdots f_{i_j}^{d_j}f_{i_{j-1}}^* \cdots f_{i_1}^*(b_\lambda),
\]
and let $x' = f_{i_{j-1}}^* \cdots f_{i_1}^*(b_\lambda)$. We claim that $f_{i_j}^*(x') \not\in \Ecal(\{ x \})$. If it were, then $f_{i_j}^*(x') = e_{a_1}^{c_1} \cdots e_{a_\ell}^{c_\ell}(x)$ for some $a_1, \ldots, a_\ell \in I$ and $c_1, \ldots, c_\ell > 0$, so that
\[
x = f_{a_\ell}^{c_\ell} \cdots f_{a_1}^{c_1}(f_{i_j}^*(x')) = f_{a_\ell}^{c_\ell} \cdots f_{a_1}^{c_1}f_{i_j}^*f_{i_{j-1}}^*\cdots f_{i_1}^*(b_\lambda).
\]
By the minimality of $j$, though, this is impossible, so $f_{i_j}^*(x') \not\in \Ecal(x)$. But this now contradicts the assumption that $\Ecal(\{ x \})$ is extremal, since for $S$ the $i_j$-string through $x'$, 
\[
S \cap \Ecal(x) \not\in \{ \emptyset, \{ u_S \}, S \}.
\]
Thus there exist $i_1,\ldots,i_k$ such that $x = f_{i_k}^* \cdots f_{i_1}^*(b_\lambda)$, i.e. $x = b_{w \cdot \lambda}$ for $w = s_{i_k} \cdots s_{i_1}$.
\end{proof}
\begin{cor}
Assume $X \subseteq \Bcal(\lambda)$ is an extremal subset such that $X = \Ecal(\{ x \})$ for some $x \in X$. If $Y \subseteq \Bcal(\lambda)$ is another extremal subset such that $\mathrm{char}(X) = \mathrm{char}(Y)$, then $X = Y$.
\end{cor}
\begin{proof}
By Proposition \ref{prop:e-gen}, we have $x = b_{w \cdot \lambda}$ for some $w \in W^\lambda$. In particular, since $b_{w \cdot \lambda}$ is the unique element of $\Bcal(\lambda)$ of weight $w \cdot \lambda$, we must have $x \in Y$ since $\mathrm{char}(X) = \mathrm{char}(Y)$. Thus $X = \Ecal(\{ x \}) \subseteq Y$ since $Y$ satisfies Theorem \ref{thm:dem-def}(1), and since their characters coincide, we must have equality: $X = Y$.
\end{proof}

Thus extremal subsets which are $\Ecal$-generated by a single element are uniquely determined by their character. This result is not true in general for extremal subsets, however:
\begin{prop}
Extremal subsets are not uniquely determined by their character. That is, there exist $\lambda \in P^+$ and $X,Y \subseteq \Bcal(\lambda)$ extremal subsets such that $\mathrm{char}(X) = \mathrm{char}(Y)$.
\end{prop}
 A witness to this proposition is given in Example \ref{ex:no}, and the statement more or less follows from the fact that there may exist lowest weight elements in an extremal subset which do not have extremal weight.

To construct two such subsets, it is useful to recall the notion of \emph{Levi branching} of normal crystals (see, for instance, \cite{BS}). For any $J \subseteq I$, we obtain a normal crystal $\Bcal(\lambda)_J$ for the root system $\Phi_J$ given by discarding the data of $f_i, e_i, \varphi_i, \varepsilon_i$ for $i \not\in J$. It is fairly straightforward to see that extremal subsets are preserved under Levi branching: that is, if $X \subseteq \Bcal(\lambda)$ is an extremal subset and $J \subseteq I$, then letting $X_J \subseteq \Bcal(\lambda)_J$ denote the image of $X$ under Levi branching to $J$, we have that each connected component of $X_J$ is extremal.

In the special case that $\Bcal(\lambda)$ is a type $A_{n-1}$ crystal and $J = [n-2]$, the branched crystal $\Bcal(\lambda)_J$ decomposes as
\[
\Bcal(\lambda)_J \cong \bigoplus_{\mu : \lambda / \mu \text{ a horizontal strip}} \Bcal(\mu) \boxtimes \Bcal(\lambda/\mu),
\]
where $\Bcal(\mu) \boxtimes \Bcal(\lambda/\mu)$ consists of those $x \in SSYT_n(\lambda)$ such that $x\vert_{\le n-1} \in SSYT_{n-1}(\mu)$, and $\lambda / \mu$ is a horizontal strip if it does not contain multiple cells in the same column. Thus $x,y \in \Bcal(\lambda)$ are in the same connected component of $\Bcal(\lambda)_J$ if and only if their entries of $n$ are in the same positions. As a crystal, $\Bcal(\mu) \boxtimes \Bcal(\lambda/\mu)$ is related to $\Bcal(\mu)$ via the natural restriction map on the weight lattice $P_{GL_n} \to P_{GL_{n-1}}$.

\begin{example}\label{ex:no}
Consider the type $A_5$ crystal $\Bcal(\lambda)$ indexed by $\lambda = (3,1,1,0,0,0)$, so that $I = [5]$, and the following subsets of $\Bcal(\lambda)$:
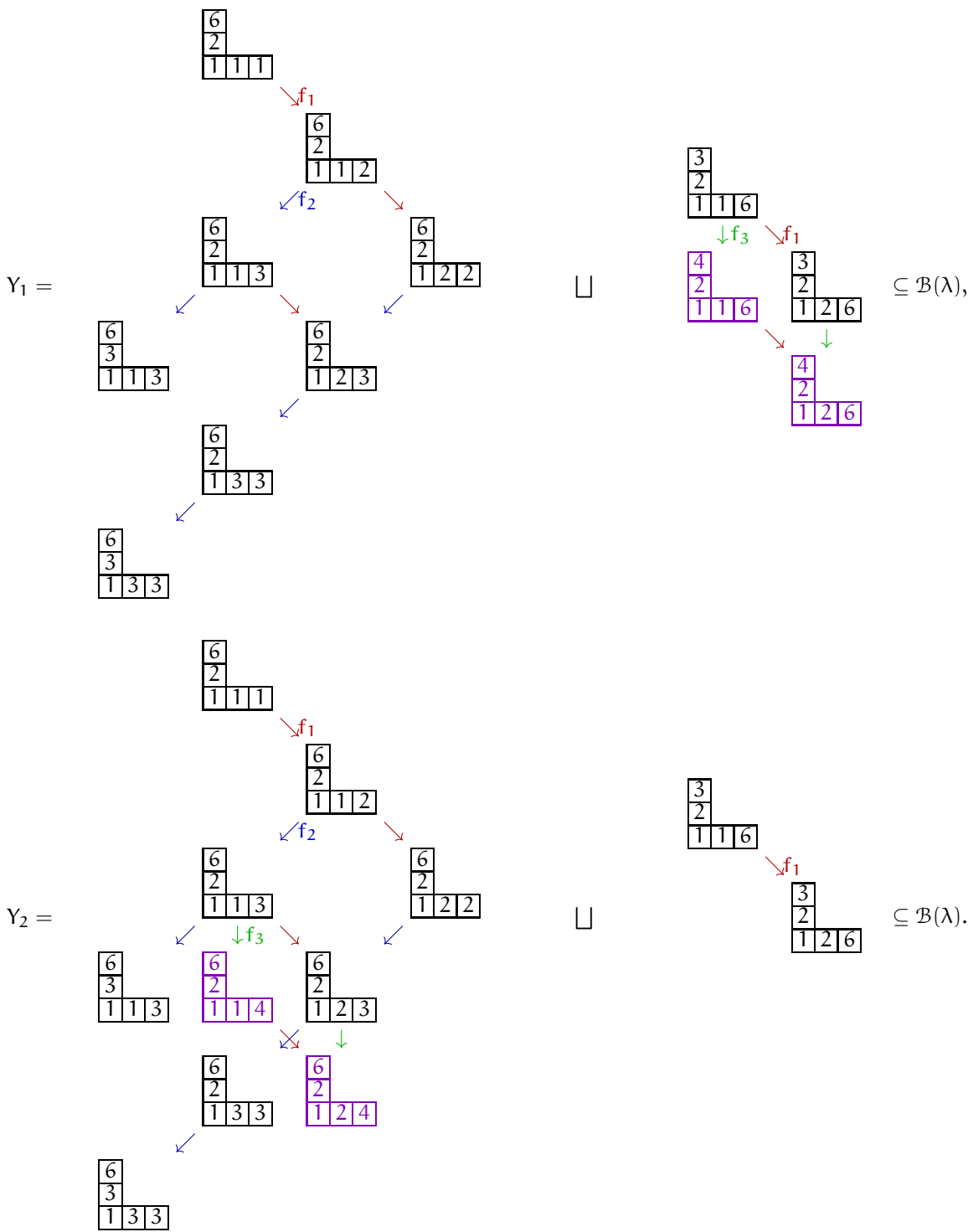
\[
\ytableausetup{boxsize=1em}
\begin{array}{cc}
\begin{tikzpicture}[scale=1.1]
\draw (-3,-3.5) node {$Y_1=$};
\draw (0,0) node (u) {$\begin{ytableau}
6 \\
2 \\
1 & 1 & 1
\end{ytableau}$};
\draw (1.5,-1.5) node (u1) {$\begin{ytableau}
6 \\
2 \\
1 & 1 & 2
\end{ytableau}$};
\draw (3,-3) node (u11) {$\begin{ytableau}
6 \\
2 \\
1 & 2 & 2
\end{ytableau}$};
\draw (0,-3) node (u12) {$\begin{ytableau}
6 \\
2 \\
1 & 1 & 3
\end{ytableau}$};
\draw (-1.5,-4.5) node (u122) {$\begin{ytableau}
6 \\
3 \\
1 & 1 & 3
\end{ytableau}$};
\draw (1.5,-4.5) node (u112) {$\begin{ytableau}
6 \\
2 \\
1 & 2 & 3
\end{ytableau}$};
\draw (0,-6) node (u1122) {$\begin{ytableau}
6 \\
2 \\
1 & 3 & 3
\end{ytableau}$};
\draw(-1.5,-7.5) node (u11222) {$\begin{ytableau}
6 \\
3 \\
1 & 3 & 3
\end{ytableau}$};
%\draw (0,-4.5) node (u123) {$\begin{ytableau}
%2 \\
%1 & 1 & 4
%\end{ytableau}$};
%\draw (1.5,-6) node (u1123) {$\begin{ytableau}
%2 \\
%1 & 2 & 4
%\end{ytableau}$};
\path[->] (u) edge[darkred,midway,right] node {$f_1$} (u1);
\path[->] (u1) edge[darkred,midway,right] (u11);
\path[->] (u12) edge[darkred,midway,right] (u112);
%\path[->] (u123) edge[darkred,midway,right] (u1123);

\path[->] (u1) edge[darkblue,midway,right] node {$f_2$} (u12);
\path[->] (u12) edge[darkblue,midway,right] (u122);
\path[->] (u11) edge[darkblue,midway,right] (u112);
\path[->] (u112) edge[darkblue,midway,right] (u1122);
\path[->] (u1122) edge[darkblue,midway,right] (u11222);

%\path[->] (u12) edge[darkgreen,midway,right] node {$f_3$} (u123);
%\path[->] (u112) edge[darkgreen,midway,right] (u1123);

\draw (5,-3.5) node {$\bigsqcup$};
\draw (7,-2) node (u12) {$\begin{ytableau}
3 \\
2 \\
1 & 1 & 6
\end{ytableau}$};
\draw (8.5,-3.5) node (u112) {$\begin{ytableau}
3 \\
2 \\
1 & 2 & 6
\end{ytableau}$};
\draw [darkpurple] (7,-3.5) node (u123) {$\begin{ytableau}
4 \\
2 \\
1 & 1 & 6
\end{ytableau}$};
\draw [darkpurple] (8.5,-5) node (u1123) {$\begin{ytableau}
4 \\
2 \\
1 & 2 & 6
\end{ytableau}$};
\path[->] (u12) edge[darkred,midway,right] node {$f_1$} (u112);
\path[->] (u123) edge[darkred,midway,right] (u1123);
\path[->] (u12) edge[darkgreen,midway,right] node {$f_3$} (u123);
\path[->] (u112) edge[darkgreen,midway,right] (u1123);

\draw (10,-3.5) node {$\subseteq \Bcal(\lambda),$};
\end{tikzpicture}
\end{array}
\]
\[
\ytableausetup{boxsize=1em}
\begin{array}{cc}
\begin{tikzpicture}[scale=1.1]
\draw (-3,-3.5) node {$Y_2=$};
\draw (0,0) node (u) {$\begin{ytableau}
6 \\
2 \\
1 & 1 & 1
\end{ytableau}$};
\draw (1.5,-1.5) node (u1) {$\begin{ytableau}
6 \\
2 \\
1 & 1 & 2
\end{ytableau}$};
\draw (3,-3) node (u11) {$\begin{ytableau}
6 \\
2 \\
1 & 2 & 2
\end{ytableau}$};
\draw (0,-3) node (u12) {$\begin{ytableau}
6 \\
2 \\
1 & 1 & 3
\end{ytableau}$};
\draw (-1.5,-4.5) node (u122) {$\begin{ytableau}
6 \\
3 \\
1 & 1 & 3
\end{ytableau}$};
\draw (1.5,-4.5) node (u112) {$\begin{ytableau}
6 \\
2 \\
1 & 2 & 3
\end{ytableau}$};
\draw (0,-6) node (u1122) {$\begin{ytableau}
6 \\
2 \\
1 & 3 & 3
\end{ytableau}$};
\draw (-1.5,-7.5) node (u11222) {$\begin{ytableau}
6 \\
3 \\
1 & 3 & 3
\end{ytableau}$};
\draw [darkpurple] (0,-4.5) node (u123) {$\begin{ytableau}
6 \\
2 \\
1 & 1 & 4
\end{ytableau}$};
\draw [darkpurple] (1.5,-6) node (u1123) {$\begin{ytableau}
6 \\
2 \\
1 & 2 & 4
\end{ytableau}$};
\path[->] (u) edge[darkred,midway,right] node {$f_1$} (u1);
\path[->] (u1) edge[darkred,midway,right] (u11);
\path[->] (u12) edge[darkred,midway,right] (u112);
\path[->] (u123) edge[darkred,midway,right] (u1123);

\path[->] (u1) edge[darkblue,midway,right] node {$f_2$} (u12);
\path[->] (u12) edge[darkblue,midway,right] (u122);
\path[->] (u11) edge[darkblue,midway,right] (u112);
\path[->] (u112) edge[darkblue,midway,right] (u1122);
\path[->] (u1122) edge[darkblue,midway,right] (u11222);

\path[->] (u12) edge[darkgreen,midway,right] node {$f_3$} (u123);
\path[->] (u112) edge[darkgreen,midway,right] (u1123);

\draw (5,-3.5) node {$\bigsqcup$};
\draw (7,-2) node (u12) {$\begin{ytableau}
3 \\
2 \\
1 & 1 & 6
\end{ytableau}$};
\draw (8.5,-3.5) node (u112) {$\begin{ytableau}
3 \\
2 \\
1 & 2 & 6
\end{ytableau}$};
%\draw (7,-3.5) node (u123) {$\begin{ytableau}
%4 \\
%2 \\
%1 & 1
%\end{ytableau}$};
%\draw (8.5,-5) node (u1123) {$\begin{ytableau}
%4 \\
%2 \\
%1 & 2
%\end{ytableau}$};
\path[->] (u12) edge[darkred,midway,right] node {$f_1$} (u112);
%\path[->] (u123) edge[darkred,midway,right] (u1123);
%\path[->] (u12) edge[darkgreen,midway,right] node {$f_3$} (u123);
%\path[->] (u112) edge[darkgreen,midway,right] (u1123);

\draw (10,-3.5) node {$\subseteq \Bcal(\lambda).$};
\end{tikzpicture}
\end{array}
\]
Note that $\mathrm{char}(Y_1) = \mathrm{char}(Y_2)$ (the purple coloring indicates those elements which are not in $Y_1 \cap Y_2$). The connected component $X$ of $\Bcal(\lambda)_J$ containing $b_\lambda$ is isomorphic to $\Bcal(3,1,1,0,0)$, consisting of those $SSYT$ with entries $\le 5$, and is thus extremal when considered as a subset of the full crystal $\Bcal(\lambda)$. Furthermore, for any $y \in Y_1 \cup Y_2$ we have $e_4(y) = 0$, $f_5(y) = 0$, and $e_5(y) \in X$. Thus $X \sqcup Y_1$ and $X \sqcup Y_2$ are both (connected) extremal subsets of $\Bcal(\lambda)$, and $\mathrm{char}(X \sqcup Y_1) = \mathrm{char}(X \sqcup Y_2)$.
\end{example}
%We conjecture, however, that Schubert crystals are the unique extremal subsets with Schubert character:
%\begin{conj}
%If $X \subseteq \Bcal(\lambda)$ satisfies $\mathrm{char}(X) = \kappa_{\Ical,\lambda}$ for some Bruhat lower order ideal $\Ical$, then $X = \Bcal_\Ical(\lambda)$.
%\end{conj}

\bibliographystyle{alpha}

\bibliography{schubert_crystals}

\begin{thebibliography}{BBBG21}

\bibitem[ADG23]{ADG}
Sami Assaf, Anne Dranowski, and Nicolle Gonz\'{a}lez.
\newblock Extremal tensor products of {D}emazure crystals.
\newblock {\em Algebras and Representation Theory}, 2023.
\newblock To appear.

\bibitem[AG21]{AG}
Sami Assaf and Nicolle Gonz\'{a}lez.
\newblock Demazure crystals for specialized nonsymmetric {M}acdonald
  polynomials.
\newblock {\em J. Combin. Theory Ser. A}, 182:Paper No. 105463, 75, 2021.

\bibitem[AS18]{AS}
Sami Assaf and Anne Schilling.
\newblock A {D}emazure crystal construction for {S}chubert polynomials.
\newblock {\em Algebr. Comb.}, 1(2):225--247, 2018.

\bibitem[BBBG21]{BBBG}
Ben Brubaker, Valentin Buciumas, Daniel Bump, and Henrik P.~A. Gustafsson.
\newblock Colored five-vertex models and {D}emazure atoms.
\newblock {\em J. Combin. Theory Ser. A}, 178:Paper No. 105354, 48, 2021.

\bibitem[BS17]{BS}
Daniel Bump and Anne Schilling.
\newblock {\em Crystal bases}.
\newblock World Scientific Publishing Co. Pte. Ltd., Hackensack, NJ, 2017.
\newblock Representations and combinatorics.

\bibitem[Dem74]{Dem74a}
Michel Demazure.
\newblock D\'{e}singularisation des vari\'{e}t\'{e}s de {S}chubert
  g\'{e}n\'{e}ralis\'{e}es.
\newblock {\em Ann. Sci. \'{E}cole Norm. Sup. (4)}, 7:53--88, 1974.

\bibitem[JL20]{JL}
Nicolas Jacon and C\'{e}dric Lecouvey.
\newblock Keys and {D}emazure crystals for {K}ac-{M}oody algebras.
\newblock {\em J. Comb. Algebra}, 4(4):325--358, 2020.

\bibitem[Jos85]{Jos}
A.~Joseph.
\newblock On the {D}emazure character formula.
\newblock {\em Ann. Sci. \'{E}cole Norm. Sup. (4)}, 18(3):389--419, 1985.

\bibitem[Kas90]{Kas90}
Masaki Kashiwara.
\newblock Crystalizing the {$q$}-analogue of universal enveloping algebras.
\newblock {\em Comm. Math. Phys.}, 133(2):249--260, 1990.

\bibitem[Kas93]{Kas93}
Masaki Kashiwara.
\newblock The crystal base and {L}ittelmann's refined {D}emazure character
  formula.
\newblock {\em Duke Math. J.}, 71(3):839--858, 1993.

\bibitem[Kas02]{Kas02}
Masaki Kashiwara.
\newblock {\em Bases cristallines des groupes quantiques}, volume~9 of {\em
  Cours Sp\'{e}cialis\'{e}s [Specialized Courses]}.
\newblock Soci\'{e}t\'{e} Math\'{e}matique de France, Paris, 2002.

\bibitem[KN94]{KN94}
Masaki Kashiwara and Toshiki Nakashima.
\newblock Crystal graphs for representations of the {$q$}-analogue of classical
  {L}ie algebras.
\newblock {\em J. Algebra}, 165(2):295--345, 1994.

\bibitem[Kou20]{Kouno}
Takafumi Kouno.
\newblock Decomposition of tensor products of {D}emazure crystals.
\newblock {\em J. Algebra}, 546:641--678, 2020.

\bibitem[Lit95]{Lit}
Peter Littelmann.
\newblock Crystal graphs and {Y}oung tableaux.
\newblock {\em J. Algebra}, 175(1):65--87, 1995.

\bibitem[LS90]{LS}
Alain Lascoux and Marcel-Paul Sch\"{u}tzenberger.
\newblock Keys \& standard bases.
\newblock In {\em Invariant theory and tableaux ({M}inneapolis, {MN}, 1988)},
  volume~19 of {\em IMA Vol. Math. Appl.}, pages 125--144. Springer, New York,
  1990.

\bibitem[Lus90]{Lus}
G.~Lusztig.
\newblock Canonical bases arising from quantized enveloping algebras.
\newblock {\em J. Amer. Math. Soc.}, 3(2):447--498, 1990.

\bibitem[Mas09]{Mason}
S.~Mason.
\newblock An explicit construction of type {A} {D}emazure atoms.
\newblock {\em J. Algebraic Combin.}, 29(3):295--313, 2009.

\bibitem[Pol89]{Polo}
Patrick Polo.
\newblock Vari\'{e}t\'{e}s de {S}chubert et excellentes filtrations.
\newblock Number 173-174, pages 10--11, 281--311. 1989.
\newblock Orbites unipotentes et repr\'{e}sentations, III.

\bibitem[Pun16]{Pun}
Anna~Ying Pun.
\newblock {\em On decomposition of the product of {D}emazure atoms and
  {D}emazure characters}.
\newblock ProQuest LLC, Ann Arbor, MI, 2016.
\newblock Thesis (Ph.D.)--University of Pennsylvania.

\bibitem[San20]{Santos}
Jo\~{a}o~Miguel Santos.
\newblock Symplectic keys and {D}emazure atoms in type {$C$}.
\newblock {\em S\'{e}m. Lothar. Combin.}, 84B:Art. 49, 12, 2020.

\end{thebibliography}

\Addresses

\end{document}